\documentclass[12pt,reqno]{amsart}
\usepackage[all]{xypic}
\usepackage{amsfonts}
\usepackage{mathrsfs}
\usepackage{amssymb}
\usepackage{pifont}

\usepackage{bold}
\def\unit{\Eins}

\input{epsf.sty}
\catcode `\@=11
\def\numberbysection{\@addtoreset{equation}{section}
         \renewcommand{\theequation}{\thesection.\arabic{equation}}}
\numberbysection
\def\subsubsection{\@startsection{subsubsection}{3}%
  \normalparindent{.5\linespacing\@plus.7\linespacing}{-.5em}%
  {\normalfont\bfseries}}
\newtheorem{thm}{Theorem}[section]

\newtheorem{lem}[thm]{Lemma}

\newtheorem{prop}[thm]{Proposition}

\newtheorem{cor}[thm]{Corollary}

\newtheorem{claim}[thm]{Claim}

\theoremstyle{definition}

\newtheorem{df}[thm]{Definition}

\newtheorem{rmk}[thm]{Remark}

\newtheorem{nota}[thm]{Notation}
\newtheorem{ex}[thm]{Example}

\def\CalC{{\mathcal C}}

\def\CalE{{\mathcal E}}

\def\Agg{{\mathcal A}gg}
\def\Set{{\mathcal S}et}
\def\Top{{\mathcal T}op}

\def\Crl{{\mathcal C}rl}

\def\CCyclic{\mathfrak  {C}}

\def\operads{{\mathfrak O}}
\def\props{{\mathfrak P}}
\def\properads{{\mathfrak P}^{ctd}}
\def\dioperads{{\mathfrak D}}
\def\modular{{\mathfrak M}}

\def\F{\mathcal F}
\def\FF{\mathfrak F}

\def\GG{\mathfrak G}

\def\C{\CalC}
\def\Z{{\mathbb Z}}
\def\N{{\mathbb N}}

\def\O{{\mathcal O}}
\def\P{{\mathcal P}}

\def\SS{{\mathbb S}}

\def\odo{\otimes \cdots \otimes}

\def\V{\asts}
\def\asts{{\mathcal V}}
\def\F{\clusters}
\def\clusters{{\mathcal F}}
\def\isoclusters{Iso(\F)}

\def\opcat{{\mathcal O }ps}

\def\smodcat{{\mathcal M}ods}
\def\forget{\mathit {forget}}

\def\final{\mathcal T}

\def\oper{\mathcal{O}p}
\def\opers{\opcat}

\def\fops{\F\text{-}\opcat}

\def\fopsc{\F\text{-}\opcat_\C}

\def\vmodsc{\V\text{-}\smodcat_\C}

\newcommand{\ph}{\phi}
\newcommand{\Ob}{\text{Obj}}

\newcommand{\Iso}{\text{Iso}}

\newcommand{\Fe}{\mathfrak{F}}
\newcommand{\io}{\iota}
\newcommand{\Op}{\mathcal{O}}
\newcommand{\iopair}{\iota_{dec\Op}}
\newcommand{\Vpair}{\V_{dec\Op}}
\newcommand{\Vpairtwo}{\V'_{dec f_{\ast}(\Op)}}
\newcommand{\Fpair}{\F_{dec\Op}}
\newcommand{\Fpairtwo}{\F'_{dec f_{\ast}(\Op)}}
\newcommand{\Fepair}{\Fe_{dec\Op}}
\newcommand{\Fepairtwo}{\Fe'_{dec f_{\ast}(\Op)}}
\newcommand{\Vpairten}{\V_{dec\Op}^{\otimes}}
\newcommand{\iopairten}{\iota_{dec\Op}^{\otimes}}
\newcommand{\Comma}[2]{(#1 \downarrow #2)}

\def\FFdeco{\FF_{dec\O}}
\def\Fdeco{\F_{dec\O}}
\def\Vdeco{\V_{dec\O}}
\def\ideco{\iota_{dec\O}}

\def\Po{\P}
\begin{document}

\title[Decorated Feynman Categories]{Decorated Feynman Categories}

\author
[Ralph M.\ Kaufmann]{Ralph M.\ Kaufmann}
\email{rkaufman@math.purdue.edu}

\address{\textsc{Purdue University Department of Mathematics,
 West Lafayette, IN 47907}
 and Max--Planck--Institute f\"ur Mathematik, Bonn, Germany}

\author
[Jason Lucas]{Jason Lucas}
\email{lucas11@math.purdue.edu}

\address{Purdue University Department of Mathematics,
 West Lafayette, IN 47907}

\begin{abstract}
In \cite{feynman}, the new concept of Feynman categories was introduced to simplify the discussion of operad--like objects. In this present paper, we demonstrate the usefulness of this approach, by introducing the concept of decorated Feynman categories. The procedure takes a Feynman category $\mathfrak F$ and a functor $\mathcal O$ to a monoidal category to produce a new Feynman category ${\mathfrak F}_{dec {\mathcal O}}$. This in one swat explains the existence of non--sigma operads, non--sigma cyclic operads, and the non--sigma--modular operads of Markl as well as all the usual candidates simply from the category $\mathfrak G$, which is a full subcategory of the category of graphs of \cite{borisovmanin}. Moreover, we explain the appearance of terminal objects noted in \cite{marklnonsigma}. We can then easily extend this for instance to the dihedral case. Furthermore, we obtain graph complexes and all other known operadic type notions from decorating and restricting the basic Feynman category $\mathfrak G$ of aggregates of corollas. We additionally show that the construction is functorial. There are further geometric and number theoretic applications, which will follow in a separate preprint.
\end{abstract}

\maketitle


\section*{Introduction}
In \cite{feynman}, Feynman categories were introduced as a universal foundational framework for treating operations and their relations as
they appear in algebra, geometry and physics.
Building on this, we now show how a pair of a Feynman category and
a functor on it gives rise to a new Feynman category.
This rather simple sounding operation has great consequences. It allows us to generate a whole new class of examples. Part of these examples are classic or recently discovered examples. For instance, we naturally obtain non-Sigma operads, non-Sigma cyclic operads and non-Sigma modular operads by the simple observation that the Associative operad exists. Using the plus construction of \cite{feynman} this existence can be viewed by the opetopic principle to actually even underlie the construction of the Feynman category for operads and is a universal example.
Another way to look at this is the reduction of all the species to simply the category $\GG$ which is a full subcategory of the category of graphs of \cite{borisovmanin}. Indeed all examples are obtained by restricting and decorating, as we present here.
The main motivation for us was to understand the calculations of \cite{hoch1,hoch2,ochoch} in a more categorical framework. 
This is also linked to extended field theory, as we shall discuss in a subsequent paper.
The theory surprisingly has direct applications to current new constructions in operad theory. For instance, one other upshot is the explanation of several coincidences observed in \cite{marklnonsigma}. This has a direct application to geometry and topology as noted in {\it loc.\ cit.}. The fact that the non-Sigma modular operads and the Dihedral modular operads actually encode geometry goes back to the classical theory of gluing surfaces from polygons as for instance described in great detail in \cite{munkres}, where one can also find when the result of a gluing is
 orientable or not. This is an instance of a special feature of terminal objects in the formalism we discuss in this paper. Given a sequence of words as describing a surface,  one can first consolidate all the words into one word. This non--self gluing is exactly the cyclic operad structure. Then one finally needs to self--glue. This is the modular operad structure. In the Feynman category language, which we set up here, this means that the inclusion functor is a minimal extension, which by definition means that a final $\oper$ exists and pushes forward to a final $\oper$. An even more interesting and richer theory appears, when one adds geometry to the subject as in \cite{KLP,KP,Dylan,ochoch}, or more combinatorics as in \cite{turaev1,turaev2}.  
  For instance using surface considerations \cite{KP}  contains the earlier definition of a brane labelled c/o system which 
 contains that of a non--sigma modular operad as a special case.
   
  The type of decoration we discuss here will have further application to our original problem of the categorical formulation of Hochschild actions. It will also be helpful in understanding different Hopf algebra constructions which appear in number theory and mathematical physics, such as in the work of Kreimer and Brown.

   In this paper, we will concentrate on the algebraic and categorical aspects, saving the other two aspects for a subsequent paper.
As an application, we  discuss in this paper the three geometries of Kontsevich, Com, Ass and Lie. This answers a question of Willwacher: indeed there is a Feynman category for the Lie case. As  
a further application, we explain the results of \cite{marklnonsigma} in our general framework and answer the question about Dihedral operads posed in that paper.

The three main results in this paper are: 
\begin{enumerate}

\item Theorem \ref{theorem1}, which states that the decorated Feynman category exists. That is given a Feynman category $\FF=(\V,\F,\imath)$ and a strict monoidal 
functor $\O\in Fun_{\otimes}(\F,\C)$ then there is a decorated Feynman category $\FFdeco=(\Vdeco,\Fdeco,\ideco)$.
The objects of $\Fdeco$ are pairs $(X,a_X)$ with $X$ an object of 
$\F$ and $a_X$ an element of $\O(X)$.

\item
Theorem \ref{diagramthm} which establishes the commutative squares (\ref{maindiageq}) which are natural in $\O$ (\ref{reldiageq})
\begin{equation}
\xymatrix{\Fepair \ar[r]^{f^{\Op}} \ar[d]_{forget} & \Fe'_{dec\, f_{\ast}(\Op)} \ar[d]^{forget'} \\
\Fe \ar[r]^f & \Fe'}
\quad
\xymatrix{\Fepair \ar[r]^{\sigma_{dec}} \ar[d]_{f^{\Op}} & \Fe_{dec\Po} \ar[d]^{f^{\Po}} \\
\Fe'_{decf_{\ast}(\Op)} \ar[r]^{\sigma'_{dec}}& \Fe'_{decf_{\ast}(\Po)} }
\end{equation}
On the categories of monoidal functors to $\C$, we get the induced diagram of adjoint functors. 
\begin{equation}
\xymatrix{\Fpair\text{-}\opcat \ar@/^/[r]^{f^{\Op}_*}\ar@/_/[d]_{forget_*} & \F'_{dec\,f_{\ast}(\Op)}\text{-}\opcat   \ar@/^/[l]^{f^{\Op }*}\ar@/^/[d]^{forget'_*} \\
\F\text{-}\opcat \ar@/^/[r]^{f_*}\ar@/_/[u]_{forget^*} & \F'\text{-}\opcat\ar@/^/[u]^{forget'^*}\ar@/^/[l]^{f^*} }
\end{equation}
\item
Finally, Theorem \ref{decothm} shows that the decoration can be recovered by the pushforward of a final functor. I.e.\ for 
$forget:\FFdeco\to \FF:forget_*(final)=\O$.
\end{enumerate}

The simple formulation given above works in the case that the target category is Cartesian. In the non--Cartesian case there is an
extra step needed in the constructions, which we also present. The theorems hold true analogously.
Finally, in the case that $\F$ is enriched and tensored over $\C$, we give  a categorically more highbrow way of defining the
decorated category. The extra work for the non-Cartesian case is necessary for the example of decoration by $Lie$.
For the other two geometries, based on the associative and commutative operads, the $\C$ is just $\Set$.

The paper is organized as follows. In the first section, we recall the definition of Feynman categories of \cite{feynman}.
In Section 2, we define $\FFdeco$, and in Section 3 prove the second result above.
Section 4 deals with terminal objects, and contains the third Theorem. It also contains the discussion of when terminal objects push forward and proves that if they push forward under $f_*$, then they do so also under $f^{\O}_*$.
In Section 5, we give the applications to the three geometries of Kontsevich and show that all the  examples of Feynman categories for the known operadic types
can be obtained by decorating the basic Feynman category $\GG$ and restricting to subcategories. This section also contains the full discussion of the application to the work and questions of Markl.

\section*{Acknowledgements}
The initial seed of the present paper was work with Yu Tsumurra and Chris Schommer--Pries. Many of the original ideas were formalized during a visit of RK to the Simons Center for Geometry and Physics, for which we thank the Simons foundation. The idea of applications to non--Sigma modular operads came while talking to Martin Markl about his work in Prague while RK was visiting the Academy of Sciences there, whom we also would like to thank. We also thank the anonymous referees for pointing out necessary clarifications.
RK would also like to thank the Max-Planck Institute for Mathematics in Bonn for its hospitality as well as the Humbolt foundation during the visit to Dirk Kreimer, where many of the ideas were discussed and developed further.

We would like  thank Martin Markl, Dirk Keimer, Francis Brown, Vladimir Turaev  and Ben Ward for discussions, as well as the referees for the careful reading of the manuscript.

RK  thankfully  acknowledges  support  from  the  Simons foundation under
collaboration grant \# 317149.

\section*{Convention} We will be dealing with symmetric monoidal categories. To make the notation simpler, we assume that these are concrete categories.
A concrete category can be thought of as one whose objects are sets with structure and whose morphisms are functions that preserve that structure. Technically, it is a category possessing a faithful monoidal functor to the category $\Set$.
This is not necessary. Instead of using elements directly, one can as usual take the underlying category of a monoidal category as in \cite{kellybook}.
We will denote by $\Set$ the monoidal category of sets with disjoint union as the monoidal structure. $\C$ will be a  monoidal category. If necessary it is assumed
to be cocomplete and have a  monoidal product that commutes with colimits as in \cite{feynman}. 
Graphs are taken to be given as in the definition of \cite{borisovmanin}, see also the Appendix of \cite{feynman}.

\section{Feynman categories}

\subsection{Feynman categories, the definition}
We recall the setup for Feynman categories from \cite{feynman}.

Fix a  symmetric monoidal category $\F$ and let $\asts$ be a category that is a groupoid, that  is $\asts=Iso(\asts)$. Denote the free symmetric monoidal category on $\V$ by $\V^{\otimes}$.
Furthermore let $\imath\colon\V\to \F$ be a functor and let $\imath^{\otimes}$ be the induced monoidal functor $\imath^{\otimes}\colon \V^{\otimes}\to \F$.

\begin{df}
\label{commadef}
\label{feynmandef}
A triple $\FF=(\V,\F,\imath)$ of objects as above is called a Feynman category
if

\begin{enumerate}
\renewcommand{\theenumi}{\roman{enumi}}

\item (Isomorphism condition)
\label{objectcond}
The monoidal functor $\imath^{\otimes}$ induces an equivalence of symmetric monoidal categories between $\V^{\otimes}$ and $Iso(\F)$.

\item (Hereditary condition) The monoidal functor $\imath^{\otimes}$ induces an equivalence of symmetric monoidal categories between $Iso(\F\downarrow \V)^{\otimes}$ and
$Iso(\F\downarrow\F)$.
\label{morcond}

\item (Size condition) For any $\ast\in \asts$, the comma category $(\clusters\downarrow\ast)$ is
essentially small,
viz.\ it is equivalent to a small category.
\end{enumerate}

\end{df}

For an explanation see the Remark below which is made precise in the next subsection.
\begin{rmk}
The basic way to understand these axioms, which also give the link to physics and their applicability is as follows. $\V$ stands for vertices. These are the  basic objects together with their symmetries. Condition (i) then says that all objects of $\F$ can be decomposed into tensor products of basic objects and moreover this is unique up to replacing the basic objects by isomorphic ones and permutations. Condition (ii) says that all the morphisms in $\F$ can be decomposed as tensor products of morphisms from tensor products of basic objects to just one basic object, i.e.\ a many to many morphism is a tensor product of many to one morphisms. A basic example is given by graphs. Here $\V$ are indeed the vertices and the morphisms are indexed by graphs, see the Appendix of \cite{feynman} for the graph formalism we use. The composition of morphisms is then the substitution of a collection of graphs with the correct number of external legs into the vertices of the other graph. Doing this vertex by  vertex is decomposition into the many to one morphisms, see \cite{feynman} for more details, especially the Section 2 and the Appendix.
Condition (iii) is technical and is there to ensure that certain needed constructions go through.
\end{rmk}

\begin{nota}
Given a Feynman category $\FF$ we will sometimes write $\V_{\FF}$ and $\F_{\FF}$ for the underlying groupoid and monoidal category and often take the liberty of dropping the subscripts if we have already fixed $\FF$.
\end{nota}

\begin{ex}
There are plenty of examples given in \cite{feynman} among them are the ones listed in Table \ref{table1}.
These are all built on the Feynman category $\GG=(\Crl,\Agg,\imath)$, see \cite{feynman}. This is a full subcategory of the subcategory of
graphs of Borisov--Manin \cite{borisovmanin}. The objects of $\Crl$ are graphs with one vertex and no edges whose flags are labeled by a set $S$. Such a corolla will be called $*_S$.
The important remark here is that it is not the graphs that appear as objects that play the familiar role of graphs in the usual theory. Rather 
the morphisms have underlying graphs and it is these that are the relevant ones.
The different flavors of operad--types are then given by decorating and restricting the morphisms, actually the graphs underling the morphisms.
This was noticed in \cite{feynman}, but  here we can indeed say that all the decorations  indeed correspond to decorated Feynman categories and restriction
means passing to subcategories. This is explained in detail in \S\ref{examplepar}.

\subsubsection{Mods and Ops}
Although Feynman categories are interesting objects of study in their own right, 
for many applications it is interesting to consider functors from them into another monoidal category. It is these functors that represent operads, PROPs etc.. 

\begin{df}\label{opdef1}
Let $\CalC$ be a symmetric monoidal category and $\FF=(\V,\F,\imath)$ be a Feynman category.
Consider the category of strong symmetric monoidal functors
$\fopsc:=Fun_{\otimes}(\clusters,\CalC)$
which we will call $\F$--$\opers$ in $\CalC$ and a particular element will be called an $\F$-$\oper$ in $\C$. The category
of functors
$\vmodsc:=Fun(\asts,\CalC)$ will be
called $\V$-modules in $\CalC$ with an element being called a $\V$--mod in $\C$.

If $\CalC$ and $\FF$ are fixed,
we will only write $\opcat$ and $\smodcat$.
\end{df}

Notice that since $\isoclusters$ is equivalent to the free symmetric
monoidal category on $\asts$ we have an equivalence of
categories between $Fun(\asts,\CalC)$ and $Fun_{\otimes}(\isoclusters,\CalC)$.

\begin{ex}
In the theory of (pseudo)-operads, $\opcat$ is the category
of (pseudo)-operads and $\smodcat$ is the category of $\SS$--modules.
A longer list of classical notions is given in \cite{feynman}.
\end{ex}

\subsubsection{Non--$\Sigma$ version}
There is a version of non--symmetric Feynman categories.
For this one lets $\F$ be only monoidal, and not symmetric monoidal and
 $\V^{\otimes}$ is taken to be simply the free monoidal category. The rest of the axioms are the same {\it mutatis mutandis}.

\subsubsection{Explanations of Condition (i) and (ii)}
Since the definition of a Feynman category is pretty dense, we unravel the conditions a bit. This also fixes some notation used later on.

{\sc Condition {\rm (i)} and Change of base.} Due to the condition (\ref{objectcond})
for each $X\in \clusters$ there exists an isomorphism
\begin{equation}
\label{objdecompeq}
\phi_X\colon X\stackrel{\sim}{\rightarrow} \otimes_{v\in I} \imath(\ast_v) \text{ with } \ast_v\in \asts
\end{equation}
 for a finite index set $I$. Moreover,  fixing a
functor $\jmath:Iso(\F)\to \V$ which yields the equivalence, we fix a decomposition for each $X$. We will call this a choice of basis.
The decomposition (\ref{objdecompeq}) has the following property: For any two such isomorphisms (choices of basis) there is a bijection  of the two index sets $\psi\colon I\to J$ and a diagram

\begin{equation}
\xymatrix
{&\bigotimes_{v\in I} \imath(\ast_v) \ar[dd]^{\simeq \bigotimes \imath(\phi_v)}\\
X\ar[ur]^{\phi_X}_{\simeq}\ar[dr]^{\simeq}_{\phi'_X}&\\
&\bigotimes_{w\in J} \imath(\ast'_w)
}
\end{equation}
where
$\phi_v\in Hom_{\asts}(\ast_v,\ast'_{\psi(v)})$ are isomorphisms.
We call the unambiguously defined value $|I|$ the {\em length of $X$} and denote it by $|X|$.

{\sc Hereditary condition {\rm (ii)} as decompositions}\label{condii}
The hereditary condition means that the comma category $(\clusters\downarrow\asts)$ generates the morphisms in the following way.
 Any morphism $X\to X'$ in $\F$ is
part of a commutative diagram
\begin{equation}
\label{morphdecompeq}
\xymatrix
{
X \ar[rr]^{\phi}\ar[d]_{\simeq}&& X'\ar[d]^{\simeq} \\
 \bigotimes_{v\in I} X_v\ar[rr]^{\bigotimes_{v\in I}\phi_{v}}&&\bigotimes_{v\in I} \imath(\ast_v)
}
\end{equation}
where $\ast_v\in \asts$, $X_v\in \clusters$ and $\phi_v\in Hom(X_v,\imath(\ast_v))$.

 Notice that if in  (\ref{morphdecompeq}) the vertical isomorphisms are fixed, then so is the lower morphism.
Hence, a choice of basis also fixes a  particular diagram of  type (\ref{morphdecompeq}) where the $X_v$ are now each a tensor product of elements of $\imath(\asts)$.

Furthermore, given {\em any} two decompositions of a morphism according to (\ref{morcond}), it follows from the previous remark  that there is a unique isomorphism in $(\imath^{\otimes}\downarrow \imath^{\otimes})$
giving an isomorphism between the two decompositions, that is between the lower rows.

The condition of equivalence of comma categories furthermore  implies that (1) for any two such decompositions $\bigotimes_{v\in I} \phi_v$ and $\bigotimes_{v'\in I'}\phi'_{v'}$
there is a bijection
$\psi:I\to I'$ and isomorphisms
$\sigma_v:X_v\to X'_{\psi(v)}$ s.t. $P^{-1}_{\psi}\circ \bigotimes_v \sigma_v\circ \phi_v =\bigotimes \phi'_{v'}$ where
$P_{\psi}$ is the permutation corresponding to $\psi$. And (2)  that these are the only isomorphisms between morphisms.

Notice that therefore  we can further decompose the $X_v$ as $X_v\simeq \bigotimes_{w\in I_v}\imath(\ast_w)$ in (\ref{morphdecompeq}), so that
for $J= \amalg_{v\in I}I_v$: $X\simeq \bigotimes_{w\in J}\ast_w=
\bigotimes_{v\in I} \left(\bigotimes_{w\in I_v}\imath(\ast_w)\right)$. And we can also assume that this is the decomposition of $X$ if we choose a base functor $\jmath$. This means there is a diagram

\begin{equation}
\label{morphdecompeq2}
\xymatrix
{
X \ar[rr]^{\phi}\ar[d]_{\simeq}&& X'\ar[d]^{\simeq} \\
 \bigotimes_{v\in I} \left(\bigotimes_{w\in I_v}\imath(\ast_w)\right)
 \ar[rr]^{\bigotimes_{v\in I}\phi_{v}}&&\bigotimes_{v\in I} \imath(\ast_v)
}
\end{equation}
with $\phi_v:\bigotimes_{w\in I_v}\imath(\ast_w)\to \imath(\ast_v)$ and the vertical isomorphisms given by $\jmath$ up to a possible permutation.

\subsection{Simplification and Enrichments}
\label{enrichedpar}

There are basically three levels for the definition of a Feynman category. 
The shortest formulation is given in Definition \ref{commadef}. 
If the category $\F$ is set--like, i.e.\ has a faithful monoidal functor to the monoidal category of sets with disjoint union, then 
there is a slightly simpler definition, see Lemma \ref{oldcritlem}, below. This is the case for the category of sets itself, but not for linear categories, such as vector spaces. For the latter one needs Definition  \ref{feynmandef}.
Finally, to  pass on to the enriched case a rather high brow categorical formulation can be used. For the more details on the enriched version; see \S\ref{enrichedpar}.
We will present the arguments on the level of Definition  \ref{feynmandef}.
As a convenience for the reader, we include a brief discussion here.

All constructions presented here go through in the enriched case as well, but we will not burden the reader with the tedious details that are straightforward for the expert.

%

\begin{lem}\cite{feynman}
\label{oldcritlem}
If $(\F,\otimes)$ has a  faithful strong symmmetric monoidal functor to $(\Set,\amalg)$, then (i) and (\ref{morphdecompeq}) imply (ii).
\end{lem}

Another way to phrase this is that for any choice of $\jmath$ the following equation holds:
\begin{multline}
Hom(X,Y)=Hom(\jmath(X),\jmath(Y))=Hom(\bigotimes_{w\in W}\imath(\ast_w),\bigotimes_{v\in v}\imath(\ast_v))\\
=
\coprod_{\text{surjections }\psi:W\twoheadrightarrow V}\prod_{v\in V}Hom(\bigotimes_{\imath(w_v)\in \psi^{-1}(v)}\imath(\ast_{w_v}),\imath(\ast_v))
\end{multline}

\subsubsection{Enrichments}
\label{enrichments}
The theory of Feynman categories also exists for enrichments. There are two cases.
The first is the Cartesian case, that is the case where the enrichment category $\CalE$ has
a monoidal product which is also Cartesian, like in $\Set$ or in $\Top$.
In this case, everything carries over verbatim if one used indexed colimits for the Kan extensions, see \cite{feynman,kellybook}. 
As explained in \cite{feynman} 
to generalize to the enriched setting the hereditary condition (\ref{morcond}) has to be  reformulated as:
\begin{itemize}

\item[(ii')] The pullback of presheaves $\imath^{\otimes \wedge}\colon [\F^{op},Set]\to [\V^{\otimes op},Set]$

{\em restricted to representable presheaves} is monoidal.
\end{itemize}

\begin{lem}
\label{altlem} 
$\FF$ is a Feynman category if and only (i), (ii') and (iii) hold.
\end{lem}

In the Cartesian case, one can use the definition that $\FF$ is a Feynman category if (i),(ii') and (iii) hold.
On the other hand, in the non--Cartesian case, that is e.g.\ for $k$-Vect some things are harder to define.  Again one uses (ii'), but
the definition of groupoid is also replaced by the condition to be a free groupoid defined by an adjunction, see \cite{feynman}.

\section{Decorated Feynman categories}
We will now give the main construction of this paper.
It takes a given Feynman category $\FF$ and an element $\O\in \F$-$\opcat_{\C}$
and renders a new Feynman category $\FFdeco$. The new Feynman category  
 has its vertices $\V$ decorated by $\O$. More precisely, 
we will define the triple  $\FFdeco=(\Vdeco,\Fdeco,\ideco)$ and then prove
that it is again a Feynman category.
We will consider two cases separately, when $\C$ is Cartesian monoidal and if it is not.

\subsection{The triple $\Fepair$}
Fix a Feynman category $\Fe = (\V,\F,\io)$, and let $\C$ be a fixed  symmetric Cartesian monoidal  category. 
\subsubsection{$\Fdeco$: Decorating $\F$ by $\O$}
For any $\Op \in \F$-$\Op ps_{\C}$,  we will define a symmetric monoidal category $\Fpair$.

The {\sc objects} in $\Fdeco$ are pairs $(X,a_X)$,
where $X \in \Ob(\F)$ and $a_X \in \Op(X)$. 

The {\sc  morphisms} in $\Fpair$ are given as follows:  $Hom_{\Fdeco}((X,a_X),(Y,a_Y))$ 
consists of those morphisms $\phi\in Hom_{\F}(X , Y)$ which satisfy $\Op(\phi)(a_X) = a_Y$. By abuse of notation, we will use the name $\phi$ also for the morphism it defines in  $\Fdeco$.
The identity morphisms are simply $id_X:(X,a_X) \to (X,a_X)$. 
Composition is well defined, since given $\phi:(X,a_X) \to (Y,a_Y)$ and $\psi:(Y,a_Y) \to (Z,a_Z)$, we have that $\Op(\psi \circ \phi)(a_X) = \Op(\psi)\circ\Op(\phi)(a_X) = \Op(\psi)(a_Y) = a_Z$. Therefore $\psi\circ\phi$ is a well-defined morphism from $(X,a_X)$ to $(Z,a_Z)$. Associativity of this composition follows from that of composition in $\F$.

We define {\sc  the monoidal product} $\otimes_{\Fpair}$ by:
$$(X,a_X) \otimes_{\Fpair} (Y,a_Y) = (X \otimes_{\F} Y, a_{X \otimes Y})$$
where $a_{X \otimes Y} = \tau(a_X \otimes_{\C} a_Y)$ and
 $\tau_{X,Y}:\Op(X)\otimes_{\C}\Op(Y) \to \Op(X \otimes_{\F} Y)$ are the natural isomorphisms provided by the strong monoidal functor $\Op$.
 On morphisms the usual underlying monoidal structure in $\F$ restricts, again by virtue of $\O$ being strong monoidal.
The identity object is $(I_{\F},e)$, where $I_{\F}$ is the identity object of $\F$ and $e$ is the distinguished element of $\Op(I_{\F})=I_{\C}$ (the single element if $I_{\C}$ is a one-point set, the identity element if $I_{\C}$ is a ground field $k$, etc.). Technically this is the element $id_{I_{\C}}\in Hom_{\C}(I_{\C},I_{\C})$.
\begin{lem} If $\lambda_X^{\F}:X \otimes I \to X$  is the left unit constraint in $\F$, then
$\Op(\lambda_X^{\F})(a_{X\otimes I}) = a_X$,
and similarly for the right unit constraint. Hence both unit constraints induce morphisms in $\Fdeco$ which we will call  $\lambda_X^{\Fdeco}$ and $\rho_X^{\Fdeco}$.
\end{lem}
\begin{proof}
 Applying $\Op$ gives the natural isomorphism $\Op(\lambda_X^{\F}):\Op(X \otimes I) \to \Op(X)$. Since $\Op$ is strong symmetric monoidal, it must preserve the left identity isomorphism, that is $\Op(\lambda_X^{\F}) = \lambda_{\Op(X)}^{\C} \circ \tau_{X,I}^{-1}:\Op(X) \otimes \Op(I) \to \Op(X)$, where $\lambda^{\C}$ is the left unit constraint in $\C$. In sum, we have
$$\Op(\lambda_X^{\F})(a_{X\otimes I})=\lambda_{\Op(X)}^{\C}\circ \tau_{X,I}^{-1}(a_{X\otimes I}) = \lambda_{\Op(X)}^{\C}(a_X \otimes e) = a_X$$
\end{proof}
Since $\O$ is a strong symmetric monoidal functor, the commutativity constraints  
$c^{\C}_{\O(X),\O(Y)}: \O(X)\otimes \O(Y)\to \O(Y)\otimes \O(X)$  are compatible
with those in $\F$ via
$c^{\C}_{\O(X),\O(Y)}=\tau_{Y,X}^{-1}\circ \O(c^{\F}_{X,Y})\circ \tau_{X,Y}$.

\begin{lem} The $\O(c^{\F})$ induce a symmetric monoidal structure on $\F_{dec\O}$.
\end{lem}

\begin{proof} Explicitly the isomorphisms are given by 
\begin{multline*}
(X,a_X)\otimes_{\Fdeco} (Y,a_Y)
=(X\otimes Y, \tau(a_X\otimes a_Y))\stackrel{(c^{\F}, \O(c^{\F}))}{\longrightarrow} 
(Y\otimes X,\O(c^{\F})\circ\tau(a_X\otimes a_Y))\\
= 
(Y\otimes X,\tau\circ c^{\C}(a_X\otimes a_Y))
= (Y\otimes X,\tau (a_Y\otimes a_X))
=(Y,a_Y)\otimes (X,a_X)
\end{multline*}
The braid relations and all compatibilities are clear.
\end{proof}

Summing up, we have the following proposition, whose proof is straightforward.
\begin{prop}
The category $\Fdeco$ is symmetric monoidal for $\otimes_{\Fdeco}$, $ I_{\Fdeco}$, $\lambda_X^{\Fdeco}$, $\rho_X^{\Fdeco}$ and the induced associativitiy and commutativity constraints. \qed
\end{prop}

\subsubsection{Decorating $\V$ by $\O$, $\Vdeco$ and the functor $\ideco$}
We take as the {\sc objects} of $\Vpair$ pairs $(\ast,a_{\ast})$ with $\ast \in \Ob(\V)$ and $a_{\ast} \in \Op(\io(\ast))$. Similar to the above, we define a morphism $\phi:(\ast_v,a_{\ast_v}) \to (\ast_w,a_{\ast_w})$ to be a morphism $\phi:\ast_v \to \ast_w$ in $\V$ such that $\Op(\io(\phi))(a_{\ast_v}) = a_{\ast_w}$. Identity morphisms and compositions for this category are defined as in $\Fpair$. Since $\phi:\ast_v \to \ast_w$ in $\V$ is an isomorphism and $\Op(\io)$ is a functor, $\Op(\io(\phi))$ is an isomorphism. Thus $\phi:(\ast_v,a_{\ast_v}) \to (\ast_w,a_{\ast_w})$ is an isomorphism, and so $\Vpair$ is a groupoid.

Finally we define a functor $\iopair:\Vpair \to \Fpair$ by $\iopair(\ast,a_{\ast}) = (\io(\ast),a_{\ast})$. For a morphism $\phi:(\ast_v,a_{\ast_v}) \to (\ast_w,a_{\ast_w})$, we take $\iopair(\phi) = \io(\phi):(\io(\ast_v),a_{\ast_v}) \to (\io(\ast_w),a_{\ast_w})$. It is clear that $\iopair$ is a functor, since $\io$ is.

\subsubsection{The Feynman category $\FFdeco$}
\begin{thm}
\label{theorem1}
$\Fepair = (\Vpair,\Fpair,\iopair)$ is a Feynman category.
\end{thm}
To prove this, we  check the Isomorphism condition, the Hereditary condition, and the Size condition. These are a bit technical and the reader not interested in these details may skip ahead.

\begin{rmk} This theorem holds in the Cartesian and non--Cartesian case with the modifications in the construction given in \S\ref{noncart1sec}, and those of \S\ref{noncart2sec}.

\end{rmk}
\begin{proof}

\textbf{Isomorphism Condition}: Let $\jmath:Iso(\F) \to \V^{\otimes}$ be a quasi-inverse of $\io^{\otimes}$, which we called a choice of basis. We will show that this induces a quasi--inverse $\jmath_{dec\O}$. 
Take $\otimes_{v \in V}(\ast_v,a_{\ast_v}) \in \Vpairten$ where $V$ is a finite indexing set. We have $\iopairten(\otimes_{v \in I}(\ast_v,a_{\ast_v})) = \otimes_{v \in V}(\io(\ast_v),a_{\ast_v}) \in \Fpair$. Take $X \in \Ob(\F)$. Since $\jmath$ is fixed, $X$ has a decomposition $X \cong \io^{\otimes}\jmath(X) = \otimes_{v \in V} \io(\ast_v)$. This gives a natural isomorphism between the identity functor on $\F$ and $\io^{\otimes}\jmath$. Call the components of this transformation $\xi_X:X \to \otimes_{v \in V} \io(\ast_v)$. 
Applying $\O$ we obtain morphisms $\Op(\xi_X):\Op(X) \to \Op(\otimes_{v \in V} \io(\ast_v))$. Let $a_{\otimes \io(\ast_v)} := \Op(\xi_X)(a_X)$ and $\otimes_{v \in V} a_{\io(\ast_v)} := \tau^{-1}(a_{\otimes \io(\ast_v)})$. We define $\jmath_{dec\Op}:\Fpair \to \Vpairten$ by $\jmath_{dec\Op}(X,a_X) = \otimes_{v \in V}(\ast_v,a_{\ast_v})$. Then we have
$$\iopairten \jmath_{dec\Op}(X,a_X) = \iopairten(\otimes_{v \in I}(\ast_v,a_{\ast_v})) = \otimes_{v \in I}(\io(\ast_v),a_{\ast_v}) = (\otimes_{v \in I} \io(\ast_v),a_{\otimes \io(\ast_v)}),$$
and $\xi_{(X,a_X)}:(X,a_X) \to \otimes_{v \in I}(\io(\ast_v),a_{\ast_v})$ is a component of a natural transformation from the identity on $\Fpair$ to $\iopairten \jmath_{dec\Op}$, proving the harder part of the equivalence.
The other direction follows from the same type of argument.

\textbf{Hereditary Condition}: This is proved by defining quasi--inverses $K$ and $L$. An object in $\Iso\Comma{\Fpair}{\Vpair}^{\otimes}$ is an arrow
\begin{equation}
\label{daggereq}
\otimes_{v \in I}\phi_v:\otimes_{v \in I}(X_v,a_{X_v}) \to \otimes_{v \in I}(\io(\ast_v),a_{\io(\ast_v)})
\end{equation}
such that each $\phi_v:(X_v,a_{X_v}) \to (\io(\ast_v),a_{\io(\ast_v)})$ is an isomorphism. Define $K:\Iso\Comma{\Fpair}{\Vpair}^{\otimes} \to \Iso\Comma{\Fpair}{\Fpair}$ on objects so that $K$ maps (\ref{daggereq}) to
$$\phi = \otimes_{v \in I}\phi_v:(\otimes_{v \in I}X_v,a_{\otimes X_v}) \to (\otimes \io(\ast_v),a_{\otimes \io(\ast_v)})$$
where $a_{\otimes X_v} = \tau(\otimes_{v \in I} a_{X_v})$ and similarly for $a_{\otimes \io(\ast_v)}$. This is well-defined, since
\begin{align*}
\Op(\phi)(a_{\otimes X_v}) = \Op(\otimes \phi_v)(\tau(\otimes a_{X_v})) &= \tau(\otimes \Op(\phi_v)(\otimes a_{X_v}))\\
&= \tau(\otimes \Op(\phi_v)(a_{X_v}))\\
&= \tau(\otimes a_{\io(\ast_v)})\\
&= a_{\otimes \io(\ast_v)}
\end{align*}
where the second equality follows by the naturality of $\tau$. Notice that $K$ amounts to passing the tensor product inside and viewing the result as a single object rather than a product of objects. We define $K$ so that it is the identity on morphisms of $\Iso\Comma{\Fpair}{\Vpair}^{\otimes}$.
\par
Recall from \ref{condii} that the hereditary condition means any morphism $\phi:X \to Y$ in $\F$ can be factored in a specific way. We denote this factorization $\hat{\phi}:X \to Y$. We define $L:\Iso\Comma{\Fpair}{\Fpair} \to \Iso\Comma{\Fpair}{\Vpair}^{\otimes}$ on objects so that it maps $\phi:(X,a_X) \to (Y,a_Y)$ to
$$\xymatrix{(\otimes_{v \in I}\io(\ast_v),\tau(\otimes_{v \in I} a_{\io(\ast_v)})) \ar[r]^-{\cong}_-{\xi_X^{-1}} & (X,a_X) \ar[r]^-{\hat{\phi}} & (Y,a_Y) \ar[r]^-{\cong}_-{\xi_Y} & (\otimes_{w \in J}\io(\ast_w),\tau(\otimes_{w \in J}a_{\io(\ast_w)}))}.$$
Here $\xi$ is as in our proof of the isomorphism condition. Denote this composition by $\tilde{\phi}:(\otimes\io(\ast_v),a_{\otimes\io(\ast_v)}) \to (\otimes\io(\ast_w),a_{\otimes\io(\ast_w)})$. On morphisms, $L$ sends the left diagram to the right diagram below:
$$\xymatrix{&&& (\otimes\io(\ast_v),a_{\otimes\io(\ast_v)}) \ar@{-->}[r]^-{\hat{f}} \ar[d]^{\cong} \ar@/_3pc/[ddd]_{\tilde{\phi}} & (\otimes\io(\ast_u),a_{\otimes\io(\ast_u)}) \ar[d]^{\cong} \ar@/^3pc/[ddd]^{\tilde{\psi}}\\
(X,a_X) \ar[r]^{f} \ar[d]_{\phi} & (X',a_{X'}) \ar[d]^{\psi} && (X,a_X) \ar[r]^-{f} \ar[d]^{\hat{\phi}} & (X',a_{X'}) \ar[d]^{\hat{\psi}}\\
(Y,a_Y) \ar[r]^{g} & (Y',a_{Y'}) && (Y,a_Y) \ar[r]^-{g} \ar[d]^{\cong} & (Y',a_{Y'}) \ar[d]^{\cong}\\
&&& (\otimes\io(\ast_w),a_{\otimes\io(\ast_w)}) \ar@{-->}[r]^-{\hat{g}} & (\otimes\io(\ast_t),a_{\otimes\io(\ast_t)})
}$$
Here $\hat{f}$ and $\hat{g}$ are the morphisms that make the top and bottom squares commute. The right hand diagram condenses to
$$\xymatrix{(\otimes\io(\ast_v),a_{\otimes\io(\ast_v)}) \ar[r]^-{\hat{f}} \ar[d]_{\tilde{\phi}} & (\otimes\io(\ast_u),a_{\otimes\io(\ast_u)}) \ar[d]^{\tilde{\psi}} \\
(\otimes\io(\ast_w),a_{\otimes\io(\ast_w)}) \ar[r]^-{\hat{g}} & (\otimes\io(\ast_t),a_{\otimes\io(\ast_t)})}$$
We must show that $K$ and $L$ are quasi-inverse to each other. Applying $KL$ to $\phi:(X,a_X) \to (Y,a_Y)$ gives $\tilde{\phi}:(\otimes\io(\ast_v),a_{\otimes\io(\ast_v)}) \to (\otimes\io(\ast_w),a_{\otimes\io(\ast_w)})$, and these arrows are isomorphic via the pair $(\xi_X,\xi_Y)$. Now applying $LK$ to $\otimes\phi_v:\otimes(X_v,a_{X_v}) \to \otimes (\io(\ast_v),a_{\io(\ast_v)})$ yields $\tilde{\phi}:(\otimes\io(\ast_v),a_{\otimes\io(\ast_v)}) \to (\otimes\io(\ast_w),a_{\otimes\io(\ast_w)})$, and these are isomorphic via $\xi$ as well.
\\\\
\textbf{Size Condition}: Fix $(\ast,a_{\io(\ast)}) \in \Vpair$. We must show that $\Comma{\Fpair}{(\ast,a_{\io(\ast)})}$ is essentially small. Let $\mathcal{A}$ be a small category that is equivalent to $\Comma{\F}{\ast}$ with equivalences $\Theta:\mathcal{A} \to \Comma{\F}{\ast}$ and $\Sigma:\Comma{\F}{\ast} \to \mathcal{A}$. An arrow in $\Comma{\Fpair}{(\ast,a_{\io(\ast)})}$ is of the form
$$\xymatrix{(X,a_X) \ar[rr]^-{f} \ar[rd]_{\phi} && (Y,a_Y) \ar[ld]^{\psi}\\
& (\io(\ast),a_{\io(\ast)}) &}$$
and from this we obtain an arrow in $\Comma{\F}{\ast}$
$$\xymatrix{X \ar[rr]^{f} \ar[rd]_{\phi} && Y \ar[ld]^{\psi}\\
& \io(\ast) &}$$
Define a category $\tilde{\mathcal{A}}$ as follows: Let
$$\Ob(\tilde{\mathcal{A}}) = \times_{\Sigma(\phi) \in \Sigma(\Ob\Comma{\F}{\ast})} \Op(\phi)^{-1}(\Op(\io(\ast))),$$
that is, pairs $(\Sigma(\phi),a_X)$ where $\Op(\phi)(a_X) = a_{\io(\ast)}$. For the morphisms of $\tilde{\mathcal{A}}$, between two pairs $(\Sigma(\phi),a_X)$ and $(\Sigma(\psi),a_Y)$, we have all morphisms $\Sigma(f)$ where $f:X \to Y$ with $\Op(f)(a_X) = a_Y$. Notice that with these definitions, the collections of objects and morphisms of $\tilde{\mathcal{A}}$ are sets, so that $\tilde{\mathcal{A}}$ is small. We can now define a functor $\tilde{\Sigma}:\Comma{\Fpair}{a_{\io(\ast)}} \to \tilde{\mathcal{A}}$ by $\tilde{\Sigma}(\phi) = (\Sigma(\phi),a_X)$ for objects where $\phi:(X,a_X) \to (\io(\ast),a_{\io(\ast)})$ and $\tilde{\Sigma}(f) = \Sigma(f)$ for morphisms. The quasi-inverse $\tilde{\Theta}:\tilde{\mathcal{A}} \to \Comma{\Fpair}{a_{\io(\ast)}}$ sends a pair $(\Sigma(\phi),a_X)$ to the morphism $\phi:(X,a_X) \to (\io(\ast),a_{\io(\ast)})$. It sends a morphisms $\Sigma(f)$ to $f$. We these definitions, it is clear that $\tilde{\Sigma}$ and $\tilde{\Theta}$ are quasi-inverse. Hence $\Comma{\Fpair}{\io(\ast)}$ is essentially small.
\end{proof}

\subsubsection{Examples}
Examples are abundant. A motivating example if $\FF=\operads$ is the Feynman category for operads and $\O$ the associative operad;
then $\FFdeco$ is the Feynman category for non--Sigma operads.
Similarly, one obtains the Feynman categories for non--Sigma cyclic and non--Sigma modular operads by decorating with the appropriate versions of the associative operad.   These are related by pushforwards as we explain in general in \S\ref{pushforwardsec}.
More details are in section \S \ref{examplepar}.

Another type of example appears if we consider Table \ref{table1}. Here the entries are obtained by two procedures, restricting the sets of morphisms and decorating.  This is dealt with in section \S\ref{decosec}.

\subsection{Non--Cartesian case}
In the non--Cartesian case, we have to be a little more careful.  There are two constructions. 

First,  if $\C$ is not Cartesian,
  everything goes through, except condition (i), since $a_X$ does not need to be decomposable.
The addition to the non--Cartesian construction is a choice of such a decomposition. If $\F$ is strict, i.e.\  $Iso(\F)=\V^{\otimes}$, this means that one simply decorates each vertex $*_v$ of $X$ separately.
In the non--strict case one has to choose a base functor $\jmath$. The construction depends on a choice of $\jmath$ only up to equivalence.

This construction is needed for $CycLie$ since operad naturally lives in $k$--$Vect$. Our definition below then recaptures the decorations used in
\cite{threegeometriesGefandseminars,ConanVogtman}.

Secondly, if $\F$ and $\C$ are enriched and satisfy further compatibilities, there is another convenient construction, that is not needed for the examples in this paper, but is general and will be used
in further study. This is outlined below.

\subsubsection{Non--Cartesian version: decoration by decomposed elements}
\label{noncart1sec}

The key observation is that after choosing $\jmath$, for $\jmath(X)=\bigotimes_{v\in I}*_v=\imath(\ast_{v_1})\odo\imath(\ast_{v_{|X|}})$
and $\phi_X:\bigotimes_{v\in I}\imath(*_v) \stackrel{\sim}{\to} X$ given by the equivalence condition (i),
 there is a chain of morphisms:

\begin{multline}
\label{ncartdageq}
Hom_{\C}(\unit,\O(\imath(*_{1})))\times \dots \times Hom_{\C}(\unit,\O(\imath(*_{|X|})))\stackrel{\otimes^{|X|-1}}\to\\
  Hom_{\C}(\unit\odo \unit, \bigotimes_{i=1}^{|X|} \O(\imath *_i))
    \stackrel{\mu^{|X|-1} \otimes \O(\phi_X)}{\longrightarrow}Hom_{\C}(\unit,\O(X))
\end{multline}
where $\mu:\unit\otimes \unit \to \unit$ is the multiplication induced by the unit constraint.
The second morphism is always an isomorphism. The first morphism is an isomorphism  if $\C$ is Cartesian, but not in general.
The first space is however the decoration for $\V^{\otimes}$. The elements of $\Vdeco$ are $(\ast_v,a_v\in \O(\imath(\ast_v)))$
and the free symmetric monoidal $\Vdeco^{\otimes}$ category has  objects
$\bigotimes_{v\in I} (*_v, a_v\in \O(*_v))$.
 These
are decomposable elements, with a given decomposition. 
The link to the previous discussion is  the natural map 
$(a_1,\dots,a_{|X|})\mapsto a_1\odo a_{|X|}\in \O(X)$.  
Pre--composing the decoration with this map give the composition along a morphism $X\to \imath(\ast)$ and by condition (ii) along any morphism. This map is also symmetric monoidal.
In the Cartesian case, the map is an isomorphism. 

Concretely:

\begin{df} For a fixed choice of  $\jmath$:
The objects of $\Fdeco$ are  tuples $(X, a_{v_1},\dots,a_{v_{|X|}})$, where 
for $\jmath(X)=\bigotimes_{v\in I}\ast_v=\ast_{v_1}\odo \ast_{v_{|X|}}$, $a_{v_i}\in \O(\ast_{v_i})$.

The morphism of $\Fdeco$ are given by the subset 
$$Hom_{\Fdeco}((X, a_{w_1},\dots,a_{w_{|X|}}),(Y,b_{v_1},\dots,b_{v_{|Y|}}))\subset Hom_{\F}(X,Y)$$
of those morphisms $\phi:X\to Y$, such that if $\bigotimes_v \phi_v$ is the decomposition of $\phi$ according to 
the diagram \eqref{morphdecompeq2}, with $\phi_v:X_v=\bigotimes_{w\in I_v}\imath(*_w)\to \imath(\ast_v)$,  then $\O(\phi_v)(\bigotimes_{w\in I_v}(a_w))=b_v$.

The  monoidal structure is given by $$(X, a_{w_1},\dots,a_{w_{|X|}})\otimes (Y,b_{v_1},\dots,b_{v_{|Y|}}))=
(X\otimes Y, a_{w_1},\dots,a_{w_{|X|}},b_{v_1},\dots,b_{v_{|Y|}})$$ and the commutativity constraints are given by those of $\F$ on the first component and the respective permutations on the others.
\end{df}

\begin{thm}
The construction depends on $\jmath$ only up to an equivalence; that is different choices of $\jmath$ yield equivalent categories $\Fdeco$.

In the case of Cartesian $\C$. $\Fdeco$ coincides, again up to equivalence, with the previous definition. 
Furthermore   Theorem \ref{theorem1}, Theorem \ref{diagramthm} and Theorem \ref{decothm} hold analagously.
\end{thm}

\begin{proof}
The first fact is clear. The second follows for the fact that in the Cartesian case the composite map of \eqref{ncartdageq} is an isomorphism. The theorems  
follow from a straightforward modification of their proofs.
\end{proof}

This is the construction that underlies the Lie construction in \cite{threegeometriesGefandseminars,ConanVogtman}: i.e.\ decorate every vertex of a graph
with an element of cyclic Lie.

\subsubsection{Non--Cartesian case and enrichment}
\label{noncart2sec}

To allow tensors that are possibly not pure, i.e.\ non--decomposable tensors, as decoration, e.g.\ in the case that $\C$ is $k$--linear,  assume that $\F$ is a subcategory of a monoidal category $\CalE$ and is enriched over $\CalE$. Furthermore  assume that $\C$ is tensored over $\CalE$.
Then one can  use the definition of objects
of $\Fdeco$ as $X\otimes \O(X)$, with $X$ an object of $\F$. Taking note that now $\V^{\otimes}$ is the free symmetric monoidal category tensored over $\CalE$ condition (i) will hold.
The full details are beyond the scope and aim of this paper. Nevertheless, we wish to provide a road map.

\begin{df}Assuming the conditions above,
the category $\Fdeco$ as an enriched category over $\CalE$ has objects $X\otimes \O(X)$ and formally the same set of morphisms $\F$ ---$Hom_{\Fdeco}(X\otimes \O(X),Y\otimes \O(Y)=Hom_{\F}(X,Y)$. A morphisms $\phi$ via tensoring  becomes the morphism $\phi\otimes \O(\phi)$ in  $\C$.
Its symmetric monoidal structure is given by $(X\otimes \O(X))\otimes_{\Fdeco} (Y\otimes \O(Y))=(X\otimes_\F Y)\otimes \O(X\otimes Y)$, with composition of morphisms and symmetries given by the isomorphism 
$(X\otimes Y)\otimes \O(X\otimes Y)\simeq (X\otimes \O(X))\otimes (Y\otimes \O(Y))$ in $\C$ provided by the strong symmetric monoidal structure of $\O$.
 $\Vdeco$ is likewise defined by objects $(V \otimes \O(\imath(V))$ with $V\in \V$ and the morphisms of $\V$. The inclusion is given by 
 $\imath (V\otimes \O(\imath (V)))=(\imath(V)\otimes \O(\imath(V)))$.
\end{df}

\begin{rmk}
Passing to elements, we recover the initial construction in the Cartesian case and possibly non--pure tensors in the non--Cartesian case.
This follows from the fact that the arrows on category of elements are given by commutative diagrams.
$$
\xymatrix{
\unit \ar[r]^-{a_x} \ar[dr]_{a_y}&X\otimes (\O(X))\ar[d]^{\phi\otimes \O(\phi)}\\
&Y\otimes \O(Y)
}
$$
\end{rmk}

Although it is straightforward  to prove the following claim using the enriched version of Feynman categories, it is outside the scope of this  paper and we defer the proof to a subsequent analysis.
\begin{claim} The enriched version of 
$\Fdeco$  with $\Vdeco$ given by $(\imath(V)\otimes \O(\imath(V))$ is an enriched Feynman category. Its elements yield the construction above for the Cartesian case. 
Furthermore the  Theorem \ref{theorem1}, Theorem \ref{diagramthm} and Theorem \ref{decothm} hold analagously.
\end{claim}

\begin{ex} Say $\C=k$--$Vect$ the category of $k$--vector spaces and let $\F$ be given by graphs.
Then we can first replace $\F$ by a category inside $k$--Vect by linearization, i.e.\ replace $X$ by the one dimensional vector space $k X$
and replace the morphisms $\phi$ again by their linearization, i.e. the  free vector spaces they generate.

Replacing $(X,a_X)$ for $X=\bigotimes_{i=1}^{|X|}\ast_i$ by $(X\otimes \O(X))$, and regarding elements means that we choose an element , i.e.\ a tensor, 
in $k\ast_1\otimes_k \dots\otimes_k k\ast_{|X|}\otimes_k \O(\ast_1)\otimes_k\dots\otimes_k \O(\ast_{|X|})\simeq  X\otimes \O(X)$.

This is the usual treatment for Feynman rules. It just means that we look at sums of pure decoations and the correlation functions are linear at all vertices.
\end{ex}

\section{Functoriality Theorem}
\subsection{Push-Forwards}
\label{pushforwardsec}
We will show that decoration behaves well with pushforwards.

Let $\Fe = (\V,\F,\io)$ be a Feynman category, let $\C$ be a fixed  symmetric monoidal category, and let $\Op$ be an $\F$-$\Op ps$ in $\C$.

Recall from \cite{feynman}  that Feynman categories form a category (actually a 2--category) and there is a pushforward (left Kan extension) for $\opcat$.
This pushforward realizes for instance the modular envelope or the PROP generated by an operad.
Consider another Feynman category $\Fe'$ and a functor $(v,f):\Fe \to \Fe'$, so that $f:\F \to \F'$ is a symmetric monoidal functor. From $f$ we get a functor, the pushforward
$$f_{\ast}:\F\text{-}\Op ps_{\C} \to \F'\text{-}\Op ps_{\C}$$
via
\begin{equation}
\label{colimeq}
f_{\ast}(\Op)(X') = colim_{\Comma{f}{X'}}\Op\circ P
\end{equation}
where $\Comma{f}{X'}$ is the comma category and $P$ is the projection functor $P:\Comma{f}{X'} \to \F$ given by
$$P(X,\ph:f(X) \to X') = X$$

\begin{thm} 
\label{diagramthm}
There is a functor $f^{\O}$
which makes the following diagram commutative 
\begin{equation}
\label{maindiageq}
\xymatrix{\Fepair \ar[r]^{f^{\Op}} \ar[d]_{forget} & \Fe'_{dec\, f_{\ast}(\Op)} \ar[d]^{forget'} \\
\Fe \ar[r]^f & \Fe'}
\end{equation}
Here the vertical arrows are the forgetful functors which forget the decorations:
$$f(forget(X,a_X)) = f(X) = forget'(f(X),\mu_X(a_X)) = forget'(f^{\Op}(X,a_X))$$

This functor is natural in $\O$, that is for any morphism $\O\to \P$ in $\FF\text{-}\opcat$ there is a diagram 
\begin{equation}
\label{reldiageq}
\xymatrix{\Fepair \ar[rr]^{\sigma_{dec}} \ar[dd]_{f^{\Op}} && \Fe_{dec\Po} \ar[dd]^{f^{\Po}} \\
&&\\
\Fe'_{decf_{\ast}(\Op)} \ar[rr]^{\sigma'_{dec}} && \Fe'_{decf_{\ast}(\Po)} }
\end{equation}
\end{thm}

\begin{cor} 
\label{functorcor}
On the level of $\opcat$ we get the following diagram: 
\begin{equation}
\label{functoreq}
\xymatrix{\Fpair\text{-}\opcat \ar@/^/[r]^{f^{\Op}_*}\ar@/_/[d]_{forget_*} & \F'_{dec\,f_{\ast}(\Op)}\text{-}\opcat   \ar@/^/[l]^{f^{\Op }*}\ar@/^/[d]^{forget'_*} \\
\F\text{-}\opcat \ar@/^/[r]^{f_*}\ar@/_/[u]_{forget^*} & \F'\text{-}\opcat\ar@/^/[u]^{forget'^*}\ar@/^/[l]^{f^*} }
\end{equation}
\end{cor}
\begin{rmk}
This diagram is the full generalization of the similar diagram in \cite{marklnonsigma} to arbitrary decorations and arbitrary morphisms, not just inclusions, see \S \ref{examplepar}.
\end{rmk}
The following observation is straightforward:
\begin{prop}
\label{trivlprop}
It $\final$ is a terminal object in $\FF\text{-}\opcat$ 
then $\FF_{dec\final}=\FF$
 and (\ref{maindiageq}) is a special case of (\ref{reldiageq}). \qed
 \end{prop}
For more on terminal objects see \S\ref{terminalpar}.

\begin{proof}[Proof of Theorem \ref{diagramthm}] 

Define $f^{\Op}$ on objects as follows:
$$(X,a_X) \in \Fpair \mapsto (f(X),\mu_{X}(a_X)) \in \Fpairtwo$$
where $\mu$ is the natural transformation that is paired with the colimit object $f_{\ast}(\Op)(f(X))$, and $\mu_X$ is the specific arrow
$$\mu_X: \Op\circ P(X,id:f(X) \to f(X)) \to f_{\ast}(\Op)(f(X))$$
Notice then that $\mu_X(a_X) \in f_{\ast}(\Op)(f(X))$, and so $(f(X),\mu_X(a_X))$ is in fact an object of $\Fpairtwo$. For morphisms, we define $f^{\Op}$ by
$$\xymatrix{(X,a_X) \ar[dd]_{\ph} && (f(X),\mu_X(a_X)) \ar[dd]_{f(\ph)} \\
& \mapsto &\\
(Y,a_Y) && (f(Y),\nu_Y(a_Y))}$$

Here $\nu$ is the transformation associated to $f_{\ast}(\Op)(f(Y))$. To see that this is a viable definition, we must check that the right hand side is in fact an arrow in $\Fpairtwo$, that is, that $f(\ph)$ is an arrow from $f(X)$ to $f(Y)$ and that $f_{\ast}(\Op)(f(\ph))(\mu_X(a_X)) = \nu_Y(a_Y)$. Clearly $f(\ph)$ is an arrow from $f(X)$ to $f(Y)$. Now we have the colimit diagram

$$\xymatrix{\Op\circ P(X,f(\ph):f(X) \to f(Y)) \ar[rr]^{\Op(\ph)} \ar[rd]_{\nu_X}  \ar@{=}[ddd] && \Op \circ P(Y, id:f(Y) \to f(Y)) \ar[ld]^{\nu_Y} \\
& f_{\ast}(\Op)(f(Y)) & \\
& f_{\ast}(\Op)(f(X)) \ar@{.>}[u]& \\
\Op \circ P(f(X): id:f(X) \to f(X)) \ar[ru]^{\mu_X} &&}$$

Here the dotted arrow is $f_{\ast}(\Op)(f(\ph))$. Take $a_X \in \Op(X)$ at the bottom of the diagram. Following the arrows up gives $f_{\ast}(\Op)(f(\ph))(\mu_X(a_X))$. Following the equality, going right across the top, and then going down gives $\nu_Y(\Op(\ph)(a_X)) = \nu_Y(a_Y)$. Thus $f^{\Op}$ is well defined on arrows. That $f^{\Op}$ respects compositions and identities (and is therefore a functor) follows immediately from the fact that $f$ is a functor and the nature of composition and identities in decorated Feynman categories. That $f^{\Op}$ is monoidal again follows from the same being true for $f$ and the monoidal structure on decorated Feynman categories. Hence we have a monoidal functor $f^{\Op}:\Fpair \to \Fpairtwo$.

The construction of a functor $v^{\Op}:\Vpair \to \Vpairtwo$ is nearly identical. We must only remember that an object in $\Vpair$ is a pair $(\ast_w,a_{\ast_w})$ where $a_{\ast_w} \in \Op(\io(\ast_w))$. Thus we can define $v^{\Op}$ on objects by sending $(\ast_w,a_{\ast_w})$ to $(v(\ast_w),\mu_{\ast_w}(a_{\ast_w}))$ where $\mu_{\ast_w} = \mu_{\io(\ast_w)}$ from above (recall $\io(\ast_w)$ is an object of $\F$) and on morphisms by
$$\xymatrix{(\ast_w,a_{\ast_w}) \ar[dd]_{\ph} && (v(\ast_w),\mu_{\ast_w}(a_{\ast_w})) \ar[dd]_{v(\ph)} \\
& \mapsto &\\
(\ast_z,a_{\ast_z}) && (v(\ast_z),\nu_{\ast_z}(a_{\ast_z}))}$$
where again $\nu_{\ast_z} = \nu_{\io(\ast_z)}$ from above. The proof that $v^{\Op}$ is a well-defined functor is now identical to that of $f^{\Op}$.

The necessary compatibilities of $v^{\Op}$ and $f^{\Op}$ with all relevant structure will follow readily. Thus we have a functor between Feynman categories $(v^{\Op},f^{\Op}):\Fepair \to \Fepairtwo$. We will usually denote this functor simply by $f^{\Op}$.

Take $\Op,\Po \in \F$-$\Op ps_{\C}$ and consider a natural transformation $\sigma:\Op \to \Po$. Then $\sigma$ induces a functor $\sigma_{\F dec}:\Fpair \to \F_{dec\Po}$ via
$$(X,a_X) \mapsto (X,\sigma_X(a_X))$$
(here $\sigma_X$ is the component arrow $\sigma_X:\Op(X) \to \Po(X)$ of the natural transformation). For $\phi:(X,a_X) \to (Y,a_Y)$, we let $\sigma_{\F dec}(\phi) = \Po(\phi)$. Since
$$\Po(\phi)(\sigma_X(a_X)) = \sigma_Y(\Op(\phi)(a_X)) = \sigma_Y(a_Y)$$ 
(because $\sigma:\Op \to \Po$ is natural), $\Po(\phi)$ is in fact an arrow from $(X,\sigma_X(a_X)) = \sigma_{\F dec}(X,a_X)$ to $(Y,\sigma_Y(a_Y)) = \sigma_{\F dec}(Y,a_Y)$. By our usual abuse of notation, we will denote $\sigma_{\F dec}(\phi) = \Po(\phi)$ by simply
$$\phi:(X,\sigma_X(a_X)) \to (Y,\sigma_Y(a_Y))$$
as an arrow in $\F_{dec\Po}$. Respecting identities and composition follows from the commutativity diagrams for the naturality of $\sigma$.

The transformation $\sigma$ also induces a functor $\sigma_{\V dec}:\Vpair \to \V_{dec\Po}$. Let $\sigma_{\ast_v} := \sigma_{\io(\ast_v)}:\Op(\io(\ast_v)) \to \Po(\io(\ast_v))$. We define $\sigma_{\V dec}$ on objects of $\Vpair$ by
$$\sigma_{\V dec}(\ast_v,a_{\ast_v}) = (\ast_v,\sigma_{\ast_v}(a_{\ast_v}))$$
and on arrows by
$$\sigma_{\V dec}(\phi:(\ast_v,a_{\ast_v}) \to (\ast_w,a_{\ast_w})) = \phi:(\ast_v,\sigma_{\ast_v}(a_{\ast_v})) \to (\ast_w,\sigma_{\ast_w}(a_{\ast_w}))$$
where on the right-hand side $\phi$ follows our usual abuse of notation of the arrow  in $\V_{dec\Po}$ induced by $\phi:\ast_v \to \ast_w$ such that $\Po(\io(\phi))$ maps one decoration to the other. As above the functorial axioms follow from the diagrams for the natural transformation $\sigma$.

Now taking $\sigma_{dec} = (\sigma_{\V dec},\sigma_{\F dec}):\Fepair \to \Fe_{dec\Po}$ gives a functor between Feynman categories. The necessary compatibility conditions (e.g. with $\iopair$, $\io_{dec\Po}$) will follow readily.

We also have an induced functor $\sigma'_{dec}:\Fe'_{decf_{\ast}(\Op)} \to \Fe'_{decf_{\ast}(\Po)}$. Namely, the natural transformation $\sigma:\Op \to \Po$ induces a natural transformation $\sigma':f_{\ast}(\Op) \to f_{\ast}(\Po)$, and this second natural transformation induces the functor $\sigma'_{dec}$ in the same manner as above. The induced transformation $\sigma'$ comes from the colimit diagrams for $f_{\ast}(\Op)$ and $f_{\ast}(\Po)$. We have the diagram
$$\xymatrix{\Op\circ Pr(X,\phi:f(X) \to X') \ar[rr]^-{\Op\circ \Pr(\xi)} \ar[rd]^{\mu_X} \ar[ddd]_{\sigma_X} && \Op\circ Pr(Y,\psi:f(Y) \to X') \ar[ld]_{\mu_Y} \ar[ddd]^{\sigma_Y}\\
& f_{\ast}(\Op)(X') \ar@{.>}[ddd]^-{\sigma'_{X'}} &\\
&&\\
\Po\circ Pr(X,\phi:f(X) \to X') \ar[rr]^-{\Po\circ \Pr(\xi)} \ar[rd]^{\nu_X} && \Po\circ Pr(Y,\psi:f(Y) \to X') \ar[ld]_{\nu_Y}\\
& f_{\ast}(\Po)(X') &}$$
Here $\xi:X \to Y$ is such that
$$\xymatrix{f(X) \ar[rr]^{f(\xi)} \ar[rd]_{\phi} && f(X) \ar[ld]^{\psi}\\
& X' &}$$
The dotted vertical arrow is the component arrow $\sigma'_{X'}:f_{\ast}(\Op)(X') \to f_{\ast}(\Po)(X')$ of the natural transformation $\sigma'$.

We now claim that the diagram
$$\xymatrix{\Fepair \ar[rr]^{\sigma_{dec}} \ar[dd]_{f^{\Op}} && \Fe_{dec\Po} \ar[dd]^{f^{\Po}} \\
&&\\
\Fe'_{decf_{\ast}(\Op)} \ar[rr]^{\sigma'_{dec}} && \Fe'_{decf_{\ast}(\Po)} }$$
commutes. This follows directly from the previous colimit diagram:
$$\xymatrix{(X,a_X) \ar[rr]^{\sigma_X} \ar[dd]_{f^{\Op}} && (X,\sigma_X(a_X)) \ar[dd]^{f^{\Po}} \\
&&\\
(f(X),\mu_X(a_X)) \ar[rr]^-{\sigma'_{f(X)}} && (f(X),\sigma'_{f(X)}(\mu_X(a_X))) = (f(X),\nu_X(\sigma_X(a_X))) }$$
This diagram shows the commutativity for objects. Morphisms are similar.

\end{proof}
\section{Decorations and Terminal objects}
\label{terminalpar}
Notice that if $\C$ has a terminal object $pt$ then there is a terminal element $\final$  which is a monoidal unit
in $\F$-$\opcat_{\C}$ given by $\final(*_v)=pt$ for all elements and sending all maps to iterations of the composition $pt\otimes pt\to pt$.

\begin{thm}
\label{decothm}
If $\final$ is a terminal object for $\fops$ and $forget:\Fdeco\to \F$ is the forgetful functor,
then $forget^*(\final)$ is a terminal object for $\Fdeco\text{-}\opcat$.
We have that $forget_*forget^*({\final})=\O$.
\end{thm}
\begin{proof}
The first statement either follows by calculation or the fact that $forget^*$ is a right adjoint and preserves limits.
For the second statement we give first the proof in the case that $\C$ has a terminal object $pt$ and that $\final(*_v)=pt$ for all $*_v\in \V$.
By the first statement $\final_{dec\O}=forget^*(\final)$ is again a final $\F$-$\oper$.
In this case, one can easily calculate that  $forget_*(\final_{dec\O})=\O$ by using the definition (\ref{colimeq}). Indeed,
we notice that in the comma category over any $X$, $X\stackrel{id}{\to} X$ is a final object and thus we see that $forget_* (\final_{dec\O})(X)=
\coprod_{(X,a_x)\in \Fdeco} \final (X)=\O(X).$
The general case follows analogously.
\end{proof}

\begin{df} The trivial  $\F$-$\oper$ is $\final$ defined to be the one for which $\final(*_v)=\unit$ and all the morphisms are 
identity or the natural morphism $\unit\otimes \unit\to \unit$.
\end{df}

\begin{prop}\label{trivialprop}
Let $\final$ be the trival $\F$-$\oper$. Then Theorem \ref{decothm} holds analogously.
\end{prop}
\begin{proof}
In the case that $\C$ is concrete, this follows by using the free and forgetful monoidal functor adjunction $free:\Set {\leftarrow\atop \rightarrow} \C:forget$. Indeed
 $\unit=free(pt)$, where $pt$ is a one point set.
 In the non-concrete case this is an exercise in adapting the presented arguments. 
\end{proof}

\begin{rmk}
This explains the observations of \cite{marklnonsigma}, see \S\ref{examplepar}.
\end{rmk}

The pushforward, as seen above, does not preserve final objects in general, as it is a left adjoint and not a right adjoint. However, this can happen in special contexts.
This is an extra condition that in examples has geometric relevance, see \S\ref{applicationsec}.

\begin{df}
We call a morphism of Feynman categories $i:\FF\to \FF'$ a minimal extension over $\C$ if  $\FF$-$\opcat_{\C}$ has a
 a terminal functor $\final$ and $i_*{\final}$ is a terminal object in $\FF'$-$\opcat_{\C}$.
 \end{df}
There are two examples that appear naturally. The first is $CycCom$ and \\
$ModCycCom$ for $\CCyclic \to \modular$ and the second is the decorated version
$\forget^*(CycAss)$ and $i^{\O}_*(\forget^*(CycAss))$.
\begin{prop}
\label{minextprop}
If $f:\Fe \to \Fe'$ is a minimal extension over $\C$, then $f^{\Op}:\Fepair \to \Fepairtwo$ is as well.
\end{prop}
\begin{proof}
Since $f$ is a minimal extension over $\C$, there exists a terminal functor $\final \in \F\text{-}\Op ps_{\C}$ such that $f_{\ast}(\final)$ is terminal in $\F'\text{-}\Op ps_{\C}$. Suppose $\final_{dec\Op}$ is terminal in $\Fpair\text{-}\Op ps_{\C}$. We want to show that $f_{\ast}^{\Op}(\final_{dec\Op})$ is terminal in $\Fpairtwo\text{-}\Op ps_{\C}$. By Theorem \ref{decothm}, $forget^{\ast}(\final)$ is terminal in $\Fpair\text{-}\Op ps_{\C}$ and $forget'^{\ast}(f_{\ast}(\final))$ is terminal in $\Fpairtwo\text{-}\Op ps_{\C}$. Since terminal objects are unique up to isomorphism, it follows that $forget^{\ast}(\final) \cong \final_{dec\Op}$, and further, the proof will be complete if we can show that $f_{\ast}^{\Op}(\final_{dec\Op})$ is isomorphic to $forget'^{\ast}(f_{\ast}(\final))$. We have the diagram 
\begin{equation}
\xymatrix{\Fepair \ar[r]^{f^{\Op}} \ar[d]_{forget} & \Fe'_{dec\, f_{\ast}(\Op)} \ar[d]^{forget'} \\
\Fe \ar[r]^f & \Fe'}
\end{equation}
of Theorem \ref{diagramthm}, which gives the diagram
\begin{equation}
\xymatrix{\Fpair\text{-}\opcat_{\C} \ar@/^/[r]^{f^{\Op}_*}\ar@/_/[d]_{forget_*} & \F'_{dec\,f_{\ast}(\Op)}\text{-}\opcat_{\C}   \ar@/^/[l]^{f^{\Op }*}\ar@/^/[d]^{forget'_*} \\
\F\text{-}\opcat_{\C} \ar@/^/[r]^{f_*}\ar@/_/[u]_{forget^*} & \F'\text{-}\opcat_{\C}\ar@/^/[u]^{forget'^*}\ar@/^/[l]^{f^*} }
\end{equation}
of Corollay \ref{functorcor}. The terminal functor $\final$ lives in the bottom left-hand corner. Tracing both ways around the diagram gives
$$forget'^{\ast}(f_{\ast}(\final)) = f_{\ast}^{\Op}(forget^{\ast}(\final)).$$
As stated, $forget^{\ast}(\final) \cong \final_{dec\Op}$, and so we now have
$$forget'^{\ast}(f_{\ast}(\final)) \cong f_{\ast}^{\Op}(\final_{dec\Op}).$$
Therefore $f^{\Op}_{\ast}(\final_{dec\Op})$ is terminal in $\Fpairtwo\text{-}\Op ps_{\C}$.
\end{proof}

\begin{prop}[Minimal Extension Criterion]
\label{critprop}
Consider $i:\FF\to \FF'$, which is essentially surjective and for each $X\in \FF'$ the comma category $\Comma{i}{X}$ has a terminal object,
then $i$ is a minimal extension.
\end{prop}
\begin{proof}
Using equation (\ref{colimeq}) proves the result.
\end{proof}
\section{Examples}
\label{examplepar}

We will give a few examples. Going back to Kontsevich \cite{threegeometriesGefandseminars} and later picked up in \cite{ConanVogtman}
the examples of the cyclic operads $CycCom$, $CycAss$, and $CycLie$ are the  operads furnish  new notions of commutative geometry. The consequences were also 
further analyzed in \cite{KLP,KP,KWZ}.
The example of CycAss is in other guise discussed in detail in \cite{marklnonsigma}. We assert that the same is now possible for
the infinity versions, without much ado. Finally, we answer positively the question of Markl, whether or not there is a generalization to the dihedral case,
by simply observing that one can indeed decorate with the cyclic dihedral operad. As mentioned in the introduction, the non--Sigma cases govern
oriented surfaces while the dihedral  cases govern non--oriented surfaces.

Finally, we show that all the notions of operadic--types that have been introduced so far in the literature and treated in \cite{feynman} are 
obtained from the basic Feynman category $\GG$ of graph morphisms of aggregates of corollas by decorating as described in this paper and by restriction.

 Collecting these results, we obtain:
 \begin{prop} All the examples in tables 1 and 2 are given by decorating the Feynman category $\GG$, see \cite{feynman} \qed.
 \end{prop}

\subsection{Kontsevich's three geometries}
We use the notation of tables 1 and 2.

\subsubsection{Com, or trivially decorated}
The operad $CycCom$, the operad for cyclic commutative algebras is the terminal/trivial object in $\CCyclic$-$\opcat$. 
Thus by Proposition \ref{trivialprop},
we have that $\operads_{decCom}=\operads$. The analogous statement holds for $\CCyclic$. Indeed,
there is a forgetful functor $\operads\to\CCyclic$ and the pull--back of  $CycCom$ is $Com$ and hence 
$\CCyclic_{decCycCom}=\CCyclic$. Finally using the inclusion $i:\CCyclic\to \modular$ means that the modular envelope
$i_*(Com)$ is a modular operad. Tracing around the trivially decorated diagram (\ref{maindiageq}), we see that
this is again a final operad. Indeed this is the content of Proposition \ref{minextprop}.

\subsubsection{Ass-decorated, aka.\ Non--Sigma, aka.\ non--planar}
Likewise, we can regard the cyclic associative operad, $CycAss$.  
The pull back of $CycAss$ under $forget:\operads\to\CCyclic$ is the associative operad $Ass$.
Now $\operads_{dec\, Ass}=\operads^{pl}$ is the Feynman category for non--Sigma operads. Indeed, the 
elements of $Ass(*_s)$ are the linear orders on $S$, which means that we are dealing with planar corollas as objects.
Likewise, for the morphisms the condition that $\phi(a_X)=a_Y$ means that the trees are also planar.
The story for cyclic operads is similar $\CCyclic_{dec CycAss}=\CCyclic^{pl}$.

Things are more interesting in the modular case. In this case, we have \\
$i_*(CycAss)=:ModAss$ as a possible decoration.
Indeed using this, we recover the definition of \cite{marklnonsigma} of non--sigma modular operads, which is the special case of a brane labelled  c/o system, with trivial closed part and one color of  \cite{KP}[Appendix A.6], see also \cite{KLP}
the appendix of \cite{postnikov} and \cite{marklnonsigma} for details about the correspondence between stable or almost ribbon graphs and surfaces.
$\modular^{pl}:=\modular_{dec \, ModAss}$.

Here we can understand these constructions in a more general framework. 
First, the diagram considered in   \cite{marklnonsigma} is exactly diagram (\ref{functoreq}). Then the fact that  the non--Sigma modular envelope
of $CycAss$ is terminal is obvious from Theorem \ref{decothm} and Proposition \ref{minextprop}. The key observations
are that the terminal object of $\CCyclic^{pl}$  pushed forward is indeed $CycAss$ and the $ModAss$, which is the pushforward
of the terminal object of  $\modular^{pl}$.

\subsubsection{Lie, etc. or graph complexes}
One of the most interesting generalizations is that of Lie or
in general of Kontsevich graph complexes. Here notice that $Ass,Com$ and $Lie$ are all three cyclic operads, so that they all can be used to decorate the Feynman category for cyclic operads. For $Lie$ it is important that we can also work over $k$--Vect. 
Thus, answering a question of Willwacher, indeed there is a Feynman category for the Lie case.

To go to the case of graph complexes, one needs to first shift to the odd situation and then  take colimits as described in detail in \cite{feynman}, see especially section 6.9 of {\it loc.\ cit.}.

\subsection{Infinity versions} Another interesting example stems from the infinity versions of the Ass, Com and Lie. We will deal with this in a subsequent paper.

\subsection{Dihedral or non--oriented version} In \cite{marklnonsigma} a question was asked about the possibility of a generalization from the case of non--Sigma modular to the dihedral case. In {\it loc.\ cit.} one can find a definition of the cyclic dihedral operad $CycDihed$ first defined in \cite{Braun}.
This is by definition a cyclic operad, and hence we can regard the diagram (\ref{maindiageq}) with $\FF=\CCyclic$, $\O=CycDihed$ and 
$\FF'=\modular$. Theorem \ref{decothm} then gives all the desired features.

\subsection{Decorating and restricting $\GG$}
\label{decosec}

\subsubsection{Flag labelling, direction and roots as a decoration}
Recall that $*_S$ is the one vertex graph with flags labelled by $S$ and these are the objects of $\V=\Crl$ for $\GG$.
For any set $X$ introduce the following $\GG$-$\oper$: $X(*_S)=X^{S}$. The compositions are simply given by restricting to the target flags.

Now let the set $X$ have an involution $\bar{}:X\to X$. 
Then a natural subcategory $\FF_{decX}^{dir}$ of $\GG_{decX}$ is given by the wide subcategory whose morphisms additionally satisfy that only flags marked by 
elements $x$ and $\bar x$ are glued and then contracted. That is the underlying graph has edges whose two flags are labelled by this type of element.
 In the notation of  \cite{borisovmanin} and \cite{feynman}:  $X(f)=\overline{\imath_{\phi}(f)}$.
 If $X$ is pointed by $x_0$, there is the subcategory of $\GG_{decX}$ whose objects are those generated by $*_S$ with exactly one flag labelled by $x_0$
 and where the restriction on graphs is that for the underlying graph additionally, each edge has one flag labelled by $x_0$.

Now if $X=\Z/2\Z=\{0,1\}$ with the involution $\bar 0=1$, then we can call $0$ ``out'' and $1$ ``in'', then we obtain the category of directed graphs $\GG_{dec\Z/2Z}$.
Furthermore, if $0$ is the distinguished element, we get the rooted version. 
Examples are listed in Table \ref{table2}.

\subsubsection{Genus decoration}
Let $\N$ be the $\GG$-$\oper$ which on objects of $\V$ has constant value the natural numbers $\N(*_S)=\mathbf{N_0}$. On morphisms $\N$ is defined to behave like the genus marking.
That is for $\phi:X\to *_S$, we define $\N(\phi):\N(X)=\mathbf{N_0}^{|X|}\to \mathbf{N_0}=\N(*_S)$
 as the concatenation $\mathbf{N_0}^{|X|}\stackrel{\sum}{\to} \mathbf{N_0}\stackrel{+\bar\gamma(\phi)}{\to}\mathbf {N}_0$  where $\bar\gamma(\phi)$ equals one minus the Euler characteristic of the graph underlying $\phi$. If this graph is connected this is just 
first Betti number also sometimes called the genus. This coincides with the description in \cite{feynman}, Appendix A.
Hence, if $\FF$ is a subcategory of $\GG$, then the genus marked version  is just $\FF_{dec\N}$. 
Examples are listed in Table \ref{table2}.

\subsection{Concrete Applications}
\label{applicationsec} 
\subsubsection{Summary of results applied to non--sigma Modular case}
 We now discuss the square
 
 \begin{equation}
\label{modulardiag}
\xymatrix{\F_{dec\, CycAss}=\CCyclic^{\neg \Sigma} \ar[r]^{i^{CycAss}} \ar[d]_{forget} & \modular_{dec\, i_{\ast}(CycAss)} =\modular^{\neg \Sigma}\ar[d]^{forget} \\
\CCyclic \ar[r]^i & \modular}
\end{equation}
and point out how to find the results of \cite{marklnonsigma} which we generalized. In this paragraph, we will uss the notation of {\it loc.\ cit.}

\begin{enumerate}
\item The commutative square exists simply by Theorem \ref{theorem1}.
\item By Theorem \ref{decothm}: On the left side,  if $*_C$ is final for $\CCyclic$ and hence $forget^*(*_C)=\underline{*}_C$ 
is final for $\CCyclic^{\neg \Sigma}$ . The pushforward $forget_*(\underline{*}_C)=CycAss$.
\item Again. by Theorem \ref{decothm}: On the right side,  if $*_M$ is final for $\modular$ and hence $forget^*(*_M)=\underline{*}_M$ 
is final for $\modular^{\neg \Sigma}$. The pushforward $forget_*(\underline{*}_M)=ModAss$.
\item The inclusion $i$ is a minimal extension. This is a fact explained by basic topology.
Namely gluing together polygons in their orientation by gluing edges pairwise yields all closed oriented surfaces, see e.g.\ \cite{munkres}. In the current understanding, this procedure guarantees the condition of Proposition \ref{critprop}.
\item  By Proposition \ref{minextprop}: $i^{CycAss}$ is also a minimal extension,
which explains why indeed the pushforward of the terminal $\oper$ is up to that point is still terminal.  It also reflects the fact that not gluing all edges pairwise, but preserving orientation, does yield all surfaces with boundary.
\end{enumerate}

 \subsubsection{Summary of results applied to the dihedral case}
Since there are operad morphisms $CycCom\leftarrow CycAss\to CycDihed$, where the left is the terminal object and the right arrow is the inclusion,
there is a diagram by Theorem \ref{diagramthm}:

 \begin{equation}
\label{dihedraldiag}
\xymatrix{\modular^{Dihed}\ar[d]_{forget}&\CCyclic^{Dihed} \ar[l]_{i^{CycDihed}}\ar[d]_{forget}&\CCyclic^{\neg \Sigma} \ar[r]^{i^{CycAss}} \ar[l]_{\sigma_{dec}}\ar[d]_{forget} & \modular^{\neg \Sigma}\ar[d]_{forget} \ar@/_3pc/[lll]_{\sigma'_{dec}}\\
\modular&\CCyclic\ar[l]_i \ar@{=}[r]&\CCyclic \ar[r]^i& \modular}
\end{equation}
where as decorated Feynman categories $\modular^{Dihed}=\modular_{dec\, i_{\ast}(CycDihed)}$, $\CCyclic^{Dihed}=\\
 \F_{dec\, CycDihed}$, $\CCyclic^{\neg \Sigma}=\F_{dec\, CycAss}$
and $\modular^{\neg \Sigma}=\modular_{dec\, i_{\ast}(CycAss)} $

Again the criterion of  Proposition \ref{critprop} is satisfied also for the bottom inclusions and hence the top pushforwards also respect the terminal element. 
Geometrically, doing the gluings  with orientation switches yield also the non-orientable surfaces (with boundary).
Finally, for terminal objects Theorem \ref{decothm} applies again.
And hence indeed the push forward of the terminal $\oper$ for $\modular^{Dihed}$ is the decoration $\O=i_*(CycDihed)=Mod(CycDihed)$.

\vfill

\pagebreak

\section*{Tables}
\begin{table}[h]
\begin{tabular}{llll}
$\FF$&Feynman category for&condition on graphs additional decoration&\\
\hline
$\operads$&operads&rooted trees\\
$\operads_{mult}$&operads with mult.&b/w rooted trees.\\
$\CCyclic$&cyclic operads&trees& \\
$\GG$&unmarked nc modular operads& graphs \\
$\GG^{ctd}$&unmarked  modular operads&connected graphs \\
$\modular$&modular operads&connected + genus marking \\
$\modular^{nc,}$&nc modular operads &genus marking \\
$\dioperads$&dioperads&connected directed graphs w/o directed\\
&&loops or parallel edges\\
$\props$&PROPs&directed graphs w/o directed loops\\
$\properads$&properads&connected directed graphs w/o directed loops\\
$\dioperads^{\circlearrowleft }$&wheeled dioperads&directed graphs w/o parallel edges \\
$\props^{\circlearrowleft,ctd}$& wheeled properads&connected directed graphs w/o parallel edges \\
$\props^{\circlearrowleft}$& wheeled props &directed graphs w/o parallel edges \\
\end{tabular}
\caption{\label{table1}List of Feynman categories with conditions and decorations on the graphs}
\end{table}
\end{ex}

\begin{table}[h]
\begin{tabular}{llll}
$\FFdeco$&Feynman category for&decorating $\O$&restriction\\
\hline
$\FF^{dir}$&directed version&$\Z/2\Z$ set&edges have input and output flag\\
$\FF^{rooted}$&root&$\Z/2\Z$ set& vertices have one output flag.\\
$\FF^{genus}$&genus marked&$\N$&\\[2mm]
$\operads^{\neg\Sigma}$&non-Sigma-operads&$Ass$&\\
$\CCyclic^{\neg\Sigma}$&non-Sigma-cyclic operads&$CycAss$&\\
$\modular^{\neg\Sigma}$&non--Signa-modular&$ModAss$&\\
$\CCyclic^{dihed}$&dihedral&$Dihed$&\\
$\modular^{dihed}$&dihedral modular&$ModDihed$&
\end{tabular}
\caption{\label{table2}List of decorates Feynman categories with decorating $\Op$ and possible restriction}
\end{table}

\vfill

\pagebreak

\bibliography{fcbib}
\bibliographystyle{halpha}
\end{document}